\documentclass[12pt]{article}
\usepackage{amsthm}
\usepackage{amsmath}
\usepackage{amsfonts}
\usepackage{graphicx}
\usepackage{latexsym}
\usepackage{amssymb}
\usepackage{lmodern}

\newcommand{\field}[1]{\mathbb{#1}}

\newcommand{\Z}{\field{Z}}

\newcommand{\cA}{{\cal A}}
\newcommand{\cB}{{\cal B}}
\newcommand{\cC}{{\cal C}}
\newcommand{\cD}{{\cal D}}

\newcommand{\cS}{{\cal S}}

\newcommand{\cT}{{\cal T}}
\newcommand{\cP}{{\cal P}}

\newcommand{\cR}{{\cal R}}

\newcommand{\cN}{{\cal N}}

\newtheorem{definition}{Definition}
\newtheorem{construction}{Construction}
\newtheorem{theorem}{Theorem}
\newtheorem{lemma}{Lemma}
\newtheorem{remark}{Remark}
\newtheorem{corollary}{Corollary}

\newtheorem{example}{Example}
\newtheorem{proposition}{Proposition}

\begin{document}

\bibliographystyle{plain}

\title{
\begin{center}
New Nonuniform Group Divisible Designs \\
and Mixed Steiner Systems
\end{center}
}

\author{
{\sc Tuvi Etzion}\thanks{Department of Computer Science, Technion,
Haifa 3200003, Israel, e-mail:{\tt etzion@cs.technion.ac.il}.} \and
{\sc Yuli Tan}\thanks{School of Mathematics and Statistics,
Beijing Jiaotong University, Beijing, China,
e-mail:{\tt yl.tan@bjtu.edu.cn}.} \and
{\sc Junling Zhou}\thanks{School of Mathematics and Statistics,
Beijing Jiaotong University, Beijing, China,
e-mail:{\tt jlzhou@bjtu.edu.cn}.}}

\maketitle

\begin{abstract}
This paper considers two closely related concepts,
mixed Steiner system and nonuniform group divisible design (GDD).
The distinction between the two concepts is the minimum Hamming distance, which is required for mixed Steiner systems
but not required for nonuniform group divisible $t$-designs. In other words, it means that every mixed Steiner system is a nonuniform
GDD, but the converse is not true.
A new construction for mixed Steiner systems based on orthogonal arrays and resolvable Steiner systems is presented.
Some of the new mixed Steiner systems (also GDDs) depend on the existence of  Mersenne primes or Fermat primes.
New parameters of nonuniform GDDs derived from large sets
of H-designs (which are generalizations of GDDs) are presented, and in particular,
many nonuniform group divisible $t$-designs with $t > 3$ are introduced (for which only one family was known before).
Some GDDs are with $t > 4$, parameters for which no such design was known before.
\end{abstract}

\vspace{0.5cm}



\newpage
\section{Introduction}
\label{sec:introduction}

A {\bf \emph{Steiner system of order $n$}}, S$(t,k,n)$, is a pair $(\cN,B)$, where $\cN$ is an $n$-set
(whose elements are called \emph{points}) and $B$ is a collection
of $k$-subsets (called \emph{blocks}) of~$\cN$, such that each $t$-subset of~$\cN$ is contained
in exactly one block of $B$. A Steiner system can be represented by a binary code~$\cC$ whose codewords have length $n$ and
weight $k$. For each word $x$ of length $n$ and weight~$t$, there is exactly one codeword $c \in \cC$ for
which $d(x,c)=k-t$, where $d(y,z)$ is the Hamming distance between the words $y$ and $z$.
As a code an S$(t,k,n)$ has minimum Hamming distance $2(k-t)+2$.

There are several generalization of Steiner systems to other $t$-designs.
Group divisible designs~\cite{Han75}, H-designs were defined in~\cite{Han63},
but got the name of H-designs in~\cite{Mil74}, generalized Steiner systems~\cite{Etz97},
nonuniform group divisible designs~\cite{LKRS}, and mixed Steiner systems~\cite{Etz22,Etz25}.
All these concepts are based on set of points and blocks,
where each subset of $t$ points (subject to a certain requirements) is contained in exactly one block.
Hence, all such structures are members of the family of $t$-designs.
Some of these structures demands only containment properties, but some (generalized Steiner systems and mixed Steiner systems)
requires some minimum distance.

The goal of this paper is to consider nonuniform group divisible designs and mixed Steiner systems, two structures which differ only
in the requirement of minimum Hamming distance for mixed Steiner systems which is not required for nonuniform group divisible designs.
The rest of the paper is organized as follows. Section~\ref{sec:preliminaries} presents basic definitions and results.
In Section~\ref{sec:OA_pairs} a basic construction based only on orthogonal arrays is presented.
In Section~\ref{sec:GDDfromLS} new parameters of nonuniform group divisible designs are presented. All these $t$-designs are with $t \geq 4$,
parameters for which very few designs were constructed before and $t > 4$, parameters for which no design was constructed before.
Mixed Steiner systems with the same parameters do not exist.
Section~\ref{sec:largesets} presents a new construction for mixed Steiner systems (which is also a nonuniform GDD). The construction is based of resolution
of Steiner systems and it is best applied when there exists a Mersenne prime or a Fermat prime.

\section{Preliminaries}
\label{sec:preliminaries}

This section introduces several types of designs and codes which are important and used during our exposition.
The most important ones are the mixed Steiner systems and the group divisible designs which are the topic of this work.

\begin{definition}
\label{def:mixed}
A {\bf \emph{mixed Steiner system}} \textup{MS}$(t,k,Q)$, over the mixed alphabet $Q=\Z_{q_1} \times \Z_{q_2} \times \cdots \times \Z_{q_n}$,
is a pair $(Q,\cC)$, where $\cC$ is a set of codewords (blocks) of weight $k$, over $Q$. For each word $x$ of weight $t$ over $Q$, there
exists exactly one codeword $c \in \cC$, such that $c$ covers $x$, i.e., $d(x,c)=k-t$, and
the minimum Hamming distance of the code $\cC$ is $2(k-t)+1$. If all the $q_i$'s are equal, then the design is
a Steiner system when all the $q_i$'s equal 2 and a generalized Steiner system, whose definition follows,
when all the $q_i$'s are equal and greater than 2.
\hfill\quad $\blacksquare$
\end{definition}

If there were no requirements on the minimum distance for the code, then the structure of definition~\ref{def:mixed} is a \emph{group divisible design}.
More formally a \emph{group divisible design} GDD$(t,k,\sum_{i=1}^r n_i g_i)$ is a triple $(X,G,B)$,
where $t \geq 2$, $X$ is a set of $m=\sum_{i=1}^r n_i g_i$ points, and $G$ is a set of $n$ groups,
$n = \sum_{i=1}^r n_i$ and there are $n_i$ groups with $g_i$ points, $1 \leq i \leq r$.
Finally, $B$ is a set of blocks. Each block is a $k$-subset, such that each block contains at most one point from each group
and each $t$-subset which contains at most one point from each group is contained in exactly one block of $B$.
Such a system is a GDD of type $g_1^{n_1} g_2^{n_2} ~ \cdots ~ g_r^{n_r}$.

If all the groups in the GDD are of the same size, then the GDD is uniform and if there are at least two groups of different size, then
the GDD is nonuniform. In the original definition of a GDD all the blocks had the same size three, i,e, $k=3$, and $t=2$.
The definition later was generalized to larger $k$ and later also for $t > 2$. The design was called H-design
since the groups were considered as holes (to indicate that there is no block containing two points of the same group).
A \emph{H-design} H$(n,g,k,t)$ is a GDD$(t,k,ng)$ of type $g^n$.

The two structures of mixed Steiner systems and (nonuniform) group divisible designs are closely related. Each mixed Steiner system is
also a group divisible design, but the inverse is not correct. A mixed Steiner system has a minimum Hamming distance requirement
which is not required for a group divisible design.

A different generalization is for generalized Steiner systems.
\begin{definition}
A {\bf \emph{generalized Steiner system}} \textup{GS}$(t,k,n,q)$ is a constant-weight code~$\cC$, over $\Z_q$,
whose length is $n$, weight $k$, for each codeword, such that:
\begin{enumerate}
\item The minimum Hamming distance of $\cC$ is $2(k-t)+1$.

\item Each word $x$ of length $n$ and weight $t$ over $\Z_q$ is covered by exactly one codeword $c \in \cC$,
i.e., $d(x,c)=k-t$.
\end{enumerate}
\hfill\quad $\blacksquare$
\end{definition}

The concepts of Latin squares and orthogonal arrays which will be defined now are used
for constructions of many other combinatorial structures.

\begin{definition}
A \emph{Latin square of order $k$} is a $k \times k$ array in which each row and each column is a permutation of a given $k$-set.
Two $k \times k$ Latin squares $\cA$ and $\cB$ are \emph{orthogonal} if all the ordered pairs $( \cA(i,j),\cB(i,j) )$, $1 \leq i,j \leq k$,
are distinct.

An \emph{orthogonal array} \textup{OA}$(t,n,k)$ is a $k^t \times n$ array $\cA$ over an alphabet of size $k$ in which in any projection
of $t$~columns from $\cA$ each $t$-tuple of $Q$ appears exactly once.

A row of $\cA$ will be denoted by a sequence with $n$ elements of $Q$, e.g., $(j_0,j_1,\ldots,j_{n-1})$.
\hfill\quad $\blacksquare$
\end{definition}

Orthogonal arrays were studied from point of view of combinatorial designs and also from the point of view of
one of the most important codes, namely, MDS codes.
The existence problem of orthogonal array OA$(2,n,k)$ was extensively studied in the literature~\cite{HSS99}.
Such an array is equivalent to a set of $n-2$ pairwise orthogonal Latin squares of order $k$ (see Chapter 2 of~\cite{Etz22}).
Of a special interest is a set of $k-1$ pairwise orthogonal Latin squares of order $k$ and the following well-known results.
\begin{theorem}
\label{thm:LS_OA}
$~$
\begin{enumerate}
\item There exists a set of $k-1$ pairwise orthogonal Latin squares of order $k$ if and only if there
exists a set of $k-2$ pairwise orthogonal Latin squares of order $k$.

\item An orthogonal array \textup{OA}$(2,k+1,k)$ exists if and only if there exists a set of $k-1$ pairwise orthogonal Latin squares of order $k$.

\item An orthogonal array \textup{OA}$(2,k,k)$ exists if and only if there exists a set of $k-1$ pairwise orthogonal Latin squares of order $k$.
\end{enumerate}
\end{theorem}

Unfortunately, a set of $k-1$ pairwise orthogonal Latin squares of order $k$ is known to exist when $k$ is a prime power
and it is conjectured that it does not exist for other parameters.

The following simple proposition was proved for example in~\cite{EtZh21}.

\begin{proposition}
For any integer $t \geq 2$ and for any integer $k \geq 2$, there exists an \textup{OA}$(t-1,t,k)$.
\end{proposition}

The last definition in this section is for large sets. A large set of a combinatorial design $\cD$ is a partition of the space
into disjoint copies of $\cD$. In particular a \emph{large set} of Steiner systems S$(t,k,n)$ is a partition of all the $\binom{n}{k}$
$k$-subsets of an $n$-sets into disjoint copies of S$(t,k,n)$. No non-trivial such large set was constructed for $t > 2$. On the other hand
related large sets of H-designs or for a slightly more relaxed definition are known.

\begin{definition}
A {\bf \emph{large set of Steiner systems with multiplicity $\lambda$}}, \textup{LS}$(t,k,n;\lambda)$, is a collection of
Steiner systems \textup{S}$(t,k,n)$, over an $n$-set $Q$,
such that each $k$-subset of $Q$ is contained in exactly $\lambda$ of these copies of \textup{S}$(t,k,n)$.

A {\bf \emph{large set}} of \textup{H}$(n,g,k,t)$ with a set of points $X$ partitioned into $n$ groups of size $g$, denoted by
\textup{LH}$(n,g,k,t)$, is a partition of all the $k$-subsets of $X$, meeting each group in at most one point, into disjoint copies of \textup{H}$(n,g,k,t)$.
$~~~~~~~~$
\hfill\quad $\blacksquare$
\end{definition}

The two concepts of LS$(t,k,n;\lambda)$ and LH$(n,g,k,t)$ are closely related with the following proposition proved in~\cite{EtZh21}
and previously in~\cite{Etz96} for $t=3$ and $k=4$.

\begin{proposition}
\label{prop:LS_LH}
If there exist an \textup{OA}$(t,k,g)$ and an \textup{LS}$(t,k,n;g^{k-t})$, then there exists a large set of \textup{LH}$(n,g,k,t)$.
\end{proposition}

Some types of GDDs are always mixed Steiner systems as they always have the required minimum distance.
GDDs with two possible sizes of groups are the ones most considered and the same is in the current work
(with exception in Theorems~\ref{thm:lastNa} and~\ref{thm:lastNb}).
Systems with more than two group sizes were considered for example in~\cite{Etz98,Etz25a}.
It is very simple to verify the minimum distance of a GDD$(2,k,n+r)$ of type $1^n r^1$, where $r >1$, which implies the following result.
\begin{theorem}
\label{GDD_MS2_k}
A \textup{GDD}$(2,k,n+r)$ of type $1^n r^1$, where $r >1$ is an \textup{MS}$(2,k,\Z_2^n \times \Z_{r+1})$.
\end{theorem}
\begin{proof}
Given two codewords $(x_1,y_1)$ and $(x_2,y_2)$ in a GDD$(2,k,n+r)$ of type $1^n r^1$, where $x_1,x_2 \in \Z_2^n$
and $y_1,y_2 \in \Z_{r+1}$, distinguish between four cases.

\noindent
{\bf Case 1:} If $y_1=y_2 =0$, then $x_1$ and $x_2$ are two words of weight $k$ that cannot have two \emph{ones}
in the same two positions and hence $d(x_1,x_2) \geq 2(k-1)$.

\noindent
{\bf Case 2:} If $y_1=y_2 \neq 0$, then $x_1$ and $x_2$ are two words of weight $k-1$ that cannot any \emph{one}
in the same position and hence $d(x_1,x_2) \geq 2(k-1)$.

\noindent
{\bf Case 3:} If $y_1 \neq y_2$ and either $y_1=0$ or $y_2=0$, then $x_1$ and $x_2$ are two words, one of weight $k-1$
and one of weight $k$, that can have at most one \emph{one}
in the same position and hence $d(x_1,x_2) \geq k-1 +k-2=2(k-2)+1$.

\noindent
{\bf Case 4:} If $y_1 \neq y_2$, $y_1 \neq 0$, and $y_2 \neq 0$, then $x_1$ and $x_2$ are two words of weight $k-1$
that can have at most one \emph{one}
in the same position and hence $d((x_1,y_1),(x_2,y_2)) \geq 2(k-2)+1$.

In all four cases $d((x_1,y_1),(x_2,y_2)) \geq 2(k-2)+1$ and hence
the GDD$(2,k,n+r)$ of type $1^n r^1$ is an MS$(2,k,\Z_2^n \times \Z_{r+1})$.
\end{proof}

Theorem~\ref{GDD_MS2_k} can be easily generalized, using a similar proof, for $t > 2$, but this generalization is not used in our work.
\begin{theorem}
A \textup{GDD}$(t,k,n+r)$ of type $1^n r^1$, where $k \geq t$ is an \textup{MS}$(t,k,\Z_2^n \times \Z_{r+1})$.
\end{theorem}

Before starting our exposition we consider the representation of words (blocks) for a mixed Steiner system.
A~block of size $k$ in a Steiner system S$(t,k,n)$
is represented by a subset of size~$k$. A~codeword in a mixed Steiner system MS$(t,k,Q)$ over
$Q=\Z_{q_1} \times \Z_{q_2} \times \cdots \times \Z_{q_n}$ can be represented by a word of length $n$ and weight $k$, where the $i$-th coordinate
contains a letter from the alphabet~$\Z_{q_i}$. The codeword can be represented also as a subset of size $k$, where if a nonzero letter $\alpha$
is in the $i$-th coordinate of the codeword, then the subset contains the element $(i,\alpha)$, and as in a Steiner system the entries with
zeros are not represented. If some coordinates are over binary alphabet we can use $i$ instead of $(i,1)$ for these coordinates,
By abuse of notation we are going to use a mixed notation of sets and words
and it will not be mentioned which one is used. The used representations
will be understood from the context. The coordinates in a codeword of S$(t,k,n)$ will be sometimes
$1,2,\ldots,n$ and sometimes $0,1,\ldots,n-1$, i.e., the elements of $Z_n$, depending on the construction.
Note, that while $\Z_n = \{ 0,1,\ldots, n-1 \}$, $\Z_2^n$ consists of all binary words of length $n$.
If the system is over $\Z_2^n \times \Z_q$, the subset $\{ x_1,x_2,\ldots,x_r,(n+1,\alpha)\}$ correspond to the word in which $x_1,x_2,\ldots,x_r$
are $r$ position with \emph{ones} in the first $n$ positions and in the $(n+1)$-th position there is the nonzero symbol $\alpha$ from~$\Z_q$.
We start our exposition with the following simple results.

\begin{theorem}
\label{thm:lastNa}
If there exists a mixed Steiner system \textup{MS}$(1,k,Q \times \Z_q)$, where
$Q = \Z_{q_1} \times \cdots \times \Z_{q_{n-1}}$,
then $\sum_{i=1}^{n-1} (q_i -1) - (q -1)(k-1)$ is not negative and divisible by $k$.
\end{theorem}
\begin{proof}
There are $q -1$ nonzero alphabet letters in $\Z_q$. Each one of these alphabet letters must be contained in a separate block
of size $k$ with $k-1$ other nonzero elements and hence the blocks that contain nonzero alphabet letters
from $\Z_q$ contain $(q -1)(k-1)$ nonzero elements not from~$\Z_q$.
Hence, $\sum_{i=1}^{n-1} (q_i -1) \geq (q -1)(k-1)$, i.e., $\sum_{i=1}^{n-1} (q_i -1) - (q -1)(k-1)$ is not negative.
This difference must be divisible by $k$ since each block that does not contain nonzero elements from~$\Z_q$ contains exactly $k$ nonzero distinct elements.
\end{proof}

\begin{theorem}
\label{thm:lastNb}
If $Q = \Z_{q_1} \times \Z_{q_2} \times \cdots \times \Z_{q_{n-1}} \times \Z_{q_n}$, $q_i \leq q_j$ for $i < j$, not all the $q_i$'s
are equal, and $\sum_{i=1}^{n-1} (q_i -1) - (q_n -1)(k-1)$ is not negative and divisible by $k$, then
there exists a mixed Steiner system \textup{MS}$(1,k,Q)$ .
\end{theorem}
\begin{proof}
The proof is by induction on $\sum_{i=1}^n (q_i -1)$. Since $\sum_{i=1}^{n-1} (q_i -1) - (q_n -1)(k-1)$ is not negative and divisible by $k$,
it follows that $\sum_{i=1}^n (q_i -1)$ is divisible by $k$.

The basis is $\sum_{i=1}^n (q_i -1)=k$. Since $\sum_{i=1}^{n-1} (q_i -1) - (q_n -1)(k-1)$
is not negative divisible by~$k$, it follows that the system contains exactly one block that contains all the $q_i$'s
(which implies that each $q_i$, $1 \leq i \leq n$, equals to 2
and the mixed Steiner system is
a Steiner system S$(1,k,k)$).

For the induction hypothesis, assume that $Q' = \Z_{q'_1} \times \Z_{q'_2} \times \cdots \times \Z_{q'_{n'-1}} \times \Z_{q'_{n'}}$,
$q'_i \leq q'_j$ for $i < j$, not all the $q_i$'s are equal,
$\sum_{i=1}^{n'} (q'_i -1)=rk$, where $r \geq 1$, and there exists a mixed Steiner system MS$(1,k,Q')$.

For the induction step let $\sum_{i=1}^n (q_i -1)=(r+1)k$, where $r \geq 1$. We form a block $\{ (i,1) ~:~ n-k+1 \leq i \leq n\}$ which contains
the first nonzero elements of the $k$ largest alphabet sizes. We replace $\Z_{q_i}$ with~$\Z_{q_i -1}$, $n-k+1 \leq i \leq n$ (it can be done since we
used one nonzero element from each of these alphabets) and a new order the alphabet letters if required (this is required only if
$q_{n-k+1}=q_{n-k}$) and some of these alphabets might disappear (if some of the taken $q_i$'s for the block were equal to 2).
With the new alphabets we can use and finish with the induction hypothesis.
\end{proof}

\begin{corollary}
Assume $Q \times \Z_{q_n} = \Z_{q_1} \times \Z_{q_2} \times \cdots \times \Z_{q_{n-1}} \times \Z_{q_n}$, $q_i \leq q_j$
for $i < j$ and not all the $q_i$'s are equal.
There exists a mixed Steiner system \textup{MS}$(1,k,Q \times \Z_{q_n})$ if and only if
$\sum_{i=1}^{n-1} (q_i -1) - (q_n -1)(k-1)$ is not negative and divisible by $k$.
\end{corollary}

\section{A Construction Based on Orthogonal Arrays}
\label{sec:OA_pairs}

Constructions based on orthogonal arrays are fundamental and also provide GDDs with various distances including ones
which are mixed Steiner systems. The presented construction correct a small inconsistency in~\cite{Etz25}.

\begin{construction}
\label{const:fromOA}
Let $\cA$ be an orthogonal array \textup{OA}$(2,k,k)$, over the alphabet $\{1,2,\ldots,k\}$, and let $r$ be an integer, $1 \leq r \leq k-1$.
Let $\cS$ be the system on $\Z_2^{rk} \times \Z_{k+1}^{k-r}$ whose blocks are
$$
B_1 \triangleq \{ \{ ik+1,ik+2,\ldots, ik+k\} ~:~ 0 \leq i \leq r-1 \},
$$
$$
B_2 \triangleq \{\{ j_1, k+j_2,\ldots,(r-1)k + j_r , (rk+1,j_{r+1}), \ldots, (rk+k-r,j_k)  \} ~:~ (j_1,j_2,\ldots,j_k) \in \cA \}.
$$
\hfill\quad $\blacksquare$
\end{construction}

\begin{theorem}
\label{thm:MSfromOA}
If $r=k-1$ then the system $\cS$ of Construction~\ref{const:fromOA} is an \textup{MS}$(2,k,\Z_2^{(k-1)k} \times \Z_{k+1})$.
\end{theorem}
\begin{proof}
Each 2-subset of the form $\{ ik+m_1, ik+m_2 \}$, where $0 \leq i \leq r-1$, $1 \leq m_1 < m_2 \leq k$, is contained in a block of $B_1$.
Since $\cA$ is an OA$(2,k,k)$ over $Q=\{1,2,\ldots,k\}$ which implies that each ordered pair of $Q$ is contained exactly once
in each projection of two colums of $\cA$,
it follows that each 2-subset $\{ i_1 k+m_1, i_2 k+m_2 \}$, where
$0 \leq i_1 < i_2 \leq r-1$, $1 \leq m_1 , m_2 \leq k$, is contained in a block of $B_2$.
Finally, since $\cA$ is an OA$(2,k,k)$, it follows that each 2-subset $\{ i k+m, ((k-1)k+1,j) \}$,
where  $0 \leq i \leq r-1$, $1 \leq m \leq k$, and $1 \leq j \leq k$, is contained in exactly one block of $B_2$.

Consider the two codewords $x=(x_1,x_2,\ldots,x_{(k-1)k},\alpha)$ and $y=(y_1,y_2,\ldots,y_{(k-1)k},\beta)$. Again, by the property of $\cA$,
if $\alpha = \beta$, then $x_i =1$ implies that $y_i=0$ and $y_i =1$ implies that $x_i=0$. As a consequence $d(x,y) = 2(k-1) > 2(k-2)+1$.
If $\alpha \neq \beta$, then for at most one $i$, $1 \leq i \leq (k-1)k$ we have $x_i=y_i=1$ and hence
$d(x,y) \geq 2(k-2)+1$.

Thus, $\cS$ is an \textup{MS}$(2,k,\Z_2^{(k-1)k} \times \Z_{k+1})$.
\end{proof}

\begin{corollary}
\label{cor:MSfromOA}
If there exists an orthogonal array \textup{OA}$(2,k,k)$ then there exists a mixed Steiner system
\textup{MS}$(2,k,\Z_2^{(k-1)k} \times \Z_{k+1})$.
\end{corollary}

\begin{theorem}
\label{thm:fromOA}
If $1 \leq r < k-1$, then
the system $\cS$ of Construction~\ref{const:fromOA} is a \textup{GDD}$(2,k,rk+k(k-r)$ of type $1^{rk} k^{k-r}$ and as a code its minimum distance
is $d=k+r-2$.
\end{theorem}
\begin{proof}
In a similar way to the proof of Theorem~\ref{thm:MSfromOA}
it easily verified that each $2$-subset of $\Z_2^{rk} \times \Z_{k+1}^{k-r}$ is contained in exactly one codeword ($k$-subset)
(only one more case of pairs in the last $k-r$ coordinates have to be considered).
Consider two codewords $x=(x_1,x_2)$ and $y=(y_1,y_2)$, where $x_1$ and $y_1$ have length $rk$ and $x_2$ and $y_2$ have length $k-r$.
If $d(x_2,y_2)=k-r-1$ then in the first $rk$ positions of $x_1$ and $y_1$ all the \emph{ones} are in
different positions and hence $d(x_1,y_1) =2r$ which implies that $d(x,y)=k+r-1$.
If $d(x_2,y_2)=k-r$ then at most one \emph{one} in the first $rk$ positions of $x_1$ and $y_1$ is in the same position and
hence $d(x_1,y_1) \geq 2(r-1)$ and hence $d(x,y)=k+r-2$.
Thus, $\cS$ is a GDD$(2,k,rk+k(k-r))$ of type $1^{rk} k^{k-r}$ and as a code its minimum distance is $d = k+r-2$.
\end{proof}

\begin{corollary}
\label{cor:fromOA}
If there exists an \textup{OA}$(2,k,k)$ then, for each $r$, $1 \leq r < k-1$, there exists a \textup{GDD}$(2,k,rk+k(k-r))$ of type $1^{rk} k^{k-r}$,
with minimum distance $d=k+r-2$.
\end{corollary}

\begin{remark}
Codes with the parameters given in Theorem~\ref{thm:fromOA} and Corollary~\ref{cor:fromOA} cannot have a better minimum
distance than the one mentioned in these results.
\end{remark}

\section{GDDs from Large sets}
\label{sec:GDDfromLS}

This section is devoted to present GDD$(t,t+1,m)$s with new parameters, where $t  >3$. Only one family presented in~\cite{LKRS}
with $t=4$ was known and in this section we show that there exists many such families and even some with $t>4$.
The following proposition that was proved in~\cite{LKRS} is the key to form a GDD$(t,t+1,m)$ with $t > 3$.

\begin{proposition}
\label{prop:LKRS}
There exists a large set of \textup{GDD}$(t,t+1,gn)$ of type $g^n$ (which is an \textup{LH}$(n,g,t+1,t)$) if and only if there
exists a \textup{GDD}$(t+1,t+2,ng+h)$ of type $g^n h^1$, where $h=g(n-t)$.
\end{proposition}

A sequence of large sets of Steiner systems with multiplicity and as a consequence large sets of H-designs based
on Proposition~\ref{prop:LS_LH} were constructed in~\cite{EtZh21}. The parameters of these large sets
are given in the following propositions. All these propositions are followed by corollaries based on
Proposition~\ref{prop:LKRS} that present parameters of new nonuniform GDDs with $t \geq 4$ and some with $t > 4$.
While these parameters are based on well-known results they are very desired parameters which were not known before.

\begin{proposition}
For each $g \geq 2$ there exist an \textup{LS}$(3,4,10;g)$ and an \textup{LH}$(10,g,4,3)$.
\end{proposition}

\begin{corollary}
For each $g \geq 2$ there exists a \textup{GDD}$(4,5,10g +7g)$ of type $g^{10} (7g)^1$.
\end{corollary}

\begin{proposition}
For each $g \geq 2$, there exist an \textup{LS}$(4,5,11;g)$, an \textup{LS}$(5,6,12;g)$, an \textup{LH}$(11,g,5,4)$, and an \textup{LH}$(12,g,6,5)$, with
possible exceptions when $g \in \{3,5,7,9,11,13\}$.
\end{proposition}

\begin{corollary}
For each $g \geq 2$ there exists a \textup{GDD}$(5,6,11g+7g)$ of type $g^{11} (7g)^1$ and there exists
a \textup{GDD}$(6,7,12g+7g)$ of type $g^{12} (7g)^1$, with possible exceptions when $g \in \{3,5,7,9,11,13\}$.
\end{corollary}

\begin{proposition}
There exists an \textup{LH}$(7,g,4,3)$ if and only if $g$ is even.
\end{proposition}

\begin{corollary}
There exists a \textup{GDD}$(4,5,7g+4g)$ of type $g^7 (4g)^1$ if and only if $g$ is even.
\end{corollary}

\begin{proposition}
There exist an \textup{LS}$(3,4,14;720)$ and an \textup{LH}$(14,720,4,3)$.
\end{proposition}

\begin{corollary}
There exists a \textup{GDD}$(4,5,14 \cdot 720+11 \cdot 720)$ of type $720^{14} (11 \cdot 720)^1$.
\end{corollary}

\begin{proposition}
There exists an \textup{LH}$(5,4h,4,3)$ for any positive integer $h$.
\end{proposition}

\begin{corollary}
There exists a \textup{GDD}$(4,5,5 \cdot 4h + 8h )$ of type $(4h)^5  (8h)^1$.
\end{corollary}

\begin{proposition}
There exists an \textup{LH}$(6,g,4,3)$ if and only if $g$ is divisible by 3.
\end{proposition}

\begin{corollary}
There exists a \textup{GDD}$(4,5,6g + 3g )$ of type $g^6  (3g)^1$ if and only if $g$ is divisible by 3.
\end{corollary}

\begin{proposition}
For each $h \geq 1$ there exist an \textup{LS}$(3,4,20;9h)$ and an \textup{LH}$(20,9h,4,3)$.
\end{proposition}

\begin{corollary}
There exists a \textup{GDD}$(4,5,20 \cdot 9h + 17 \cdot 9h )$ of type $(9h)^{20}  (17 \cdot 9h)^1$ for each $h \geq 1$.
\end{corollary}

\begin{proposition}
For each $h \geq 1$ and $\ell \geq 1$, there exist an \textup{LS}$(3,4,5 \cdot 2^\ell;9h)$ and an \textup{LH}$(5 \cdot 2^\ell,9h,4,3)$.
\end{proposition}

\begin{corollary}
There exists a \textup{GDD}$(4,5,5 \cdot 2^\ell \cdot 9h + (5 \cdot 2^\ell -3) \cdot 9h )$ of type $(9h)^{5 \cdot 2^\ell}  ((5 \cdot 2^\ell -3) \cdot 9h)^1$
for each $h \geq 1$ and $\ell \geq 1$.
\end{corollary}

All the new nonuniform GDDs that were presented in this section can be mixed Steiner systems if they will have minimum distance
three, but the following result shows that this is not possible.

\begin{theorem}
If there exists an \textup{MS}$(t,t+1, \Z_{g+1}^n \times \Z_{m+1})$ (\textup{GDD}$(t,t+1,ng+m)$ of type $g^n m^1$),
then $n \geq \max \{ m+t-1,g+t-1\}$.
\end{theorem}
\begin{proof}
Let $\cS$ be an MS$(t,t+1, \Z_{g+1}^n \times \Z_{m+1})$ and consider two words
$u=\{ (1,1),(2,1),\ldots,(t-1,1), (n+1,\alpha)\}$ and $v=\{ (1,1),(2,1),\ldots,(t-1,1), (n+1,\beta)\}$, where
$\alpha, \beta \in \Z_{m+1} \setminus \{ \bf0 \}$, $\alpha \neq \beta$. Since the minimum distance of $\cS$ is three, it follows
that $u$ and $v$ are contained in two codewords of~$\cS$,
$x=\{ (1,1),(2,1),\ldots,(t-1,1), (i,a), (n+1,\alpha)\}$ and $y=\{ (1,1),(2,1),\ldots,(t-1,1), (j,b), (n+1,\beta)\}$,
where $t \leq i,j \leq n$, $i \neq j$, $1 \leq a,b \leq g$, which implies that $n-(t-1) \geq m$.

Consider now the $g$ words of weight $t$ of the form $u=\{ (1,\alpha),(2,1),\ldots,(t,1)\}$, where $\alpha \in \Z_{g+1} \setminus \{0\}$.
These $g$ codewords must be covered by $g$ distinct codewords of the form
$x_\alpha = \{ (1,\alpha),(2,1),\ldots,(t,1), (i_\alpha,j_\alpha)\}$, where $t+1 \leq i_\alpha \leq n+1$,
$j_\alpha \neq 0$, and for two such codewords $x_\beta = \{ (1,\beta),(2,1),\ldots,(t,1), (i_\beta,j_\beta)\}$
and $x_\gamma = \{ (1,\gamma),(2,1),\ldots,(t,1), (i_\gamma,j_\gamma)\}$ we have $i_\beta \neq i_\gamma$.

This implies that $n-(t-1) \geq g$ and hence $n \geq \max \{ m+t-1,g+t-1\}$.
\end{proof}

\section{Construction from Resolutions and Orthogonal Arrays}
\label{sec:largesets}

The following general construction is a generalization of a construction in~\cite{LCZ19}.
This generalization will lead to a few interesting mixed Steiner systems.


\begin{construction}
\label{const:fromLS}
Let $\cR$ be a set of $k$-subsets of $\Z_n$ in which each
$t$-subset of $\Z_n$ is contained in at most one $k$-subset of $\cR$.
Let $\cT$ be a set that contains $r$ subsets, $\cT = \{\cT_1,\cT_2,\ldots,\cT_r\}$, where each $\cT_i$ is a Steiner system S$(t-1,k-1,n)$ on $\Z_n$
and each two distinct $\cT_i$'s are disjoint. Moreover, each $t$-subset of $\Z_n$ which is not contained in a $k$-subset of $\cR$
is contained in exactly one of the $\cT_i$'s.
Let $\cS'$ be the system on $\Z_2^n \times \Z_{r+1}$ whose blocks are
$$
B_1 \triangleq \{ \{x_1,x_2,\ldots x_k \} ~:~ \{ x_1,x_2,\ldots,x_k \} \in \cR \}
$$
and
$$
B_2 \triangleq \{ \{ x_1,x_2,\ldots,x_{k-1},(n+1,i) \} ~:~ \{ x_1,x_2,\ldots,x_{k-1} \} \in \cT_i ,~~ 1 \leq i \leq r \} ~.
$$
\hfill\quad $\blacksquare$
\end{construction}

\begin{theorem}
\label{thm:fromLS}
The system $\cS'$ generated by Construction~\ref{const:fromLS} is an \textup{MS}$(t,k,\Z_2^n \times \Z_{r+1})$.
\end{theorem}
\begin{proof}
By the definition of the set $\cR$ and the $\cT_i$'s the codewords of $\cS'$ have weight $k$ and each $t$-subset of $\Z_n$ is
contained in exactly one codeword of $\cS'$.

Consider first a word $\{ x_1,x_2,\ldots,x_t\} \in \Z_n$. This word is contained exactly once in a block of $B_1$ or a block of $B_2$
since each $t$-subset is contained exactly once either in a block of $\cR$ or in a block of one of the $\cT_i$'s.

Consider now a word $\{ x_1,x_2,\ldots,x_{t-1},(n+1,i)\}$, where
$\{ x_1,x_2,\ldots,x_{t-1} \} \in \Z_n$ and $i \in \Z_{r+1} \setminus \{ \bf0 \}$. The word $\{x_1,x_2,\ldots,x_{t-1}\}$
is contained in an $(k-1)$-subset of $\cT_i$ since $\cT_i$ is a Steiner system S$(t-1,k-1,n)$ and hence
$\{x_1,x_2,\ldots,x_{t-1},(n+1,i)\}$ is contained in a block of $B_2$.

As for the minimum distance, consider first two blocks $x=\{ x_1,x_2,\ldots,x_{k-1}, (n+1,i) \}$ and $y=\{ y_1,y_2,\ldots,y_{k-1}, (n+1,j) \}$
of $\Z_2^n \times \Z_{r+1}$. If $i=j$ then $\{ x_1,x_2,\ldots,x_{k-1} \}$ and $\{ y_1,y_2,\ldots,y_{k-1} \}$ are two blocks
in $\cS_i$, a Steiner system S$(t-1,k-1,n)$ and hence $d(x,y)=2(k-1-(t-1))+2=2k-2t+2$.
If $i \neq j$ then $\{ x_1,x_2,\ldots,x_{k-1} \}$ and $\{ y_1,y_2,\ldots,y_{k-1} \}$ are two blocks that have no $t$-subset in common,
and hence $d(x,y)=2(k-1-(t-1))+1=2k-2t+1$.

Similarly two blocks $x=\{ x_1,x_2,\ldots,x_k \}$ and $y=\{ y_1,y_2,\ldots,y_k \}$ share no $t$-subset and hence
their distance is at least $d(x,y)=2(k-(t-1))=2k-2t+2$.

Two blocks $x=\{ x_1,x_2,\ldots,x_k \}$ and $y=\{ y_1,y_2,\ldots,y_{k-1}, (n+1,i) \}$ share no $t$ subset and hence
their minimum distance is at least $k+(k-1) - 2(t-1) +1=2k-2t+2$.

Similarly, the Hamming distance of two blocks $x=\{ x_1,x_2,\ldots,x_k \}$ and $y=\{ y_1,y_2,\ldots,y_k \}$
is at least $2(k-t)+2$.

Thus, the minimum Hamming distance of $\cS'$ is $2(k-t)+1$ as required.
\end{proof}

\begin{definition}
A Steiner system \textup{S}$(2,k,n)$ on $\Z_n$ is called {\bf \emph{resolvable}} if its blocks can be partitioned into $\frac{n-1}{k-1}$
disjoint copies of \textup{S}$(1,k,n)$. Each copy of \textup{S}$(1,k,n)$ is called a {\bf \emph{parallel class}} and each
point of $\Z_n$ is contained in exactly one block of the parallel class. Similarly, a parallel class in a design is
a set of blocks that contain each point exactly once, i.e., it is a Steiner system \textup{S}$(1,k,n)$.
\hfill\quad $\blacksquare$
\end{definition}

Resolvable designs attracted lot of attention during the years from the novel exposition in~\cite{HRW72}.

\begin{construction}
Let $k$ be a prime power and $\cA$ be an \textup{OA}$(2,k,k)$, on the symbols of $\Z_k$, where each row of the first $k$ rows is
a sequence with the same symbol.
Construct a system $\cS=(Q,B)$, where $Q = \Z_k \times \Z_{k-1}$ is a set with $k(k-1)$ points (for a codeword of length $k(k-1)$).
The set of blocks $B$ of $\cS$ are partitioned into three sets, $B_1$, $B_2$, and $B_3$.
$$
B_1 \triangleq\{ \{ (i,0),(i,1),\ldots,(i,k-2) \}  ~:~ 0 \leq i \leq k-1 \}~,
$$
$$
B_2 \triangleq\{ \{ (0,j),(1,j),\ldots,(k-1,j) \}  ~:~ 0 \leq j \leq k-2 \}~,
$$
$$
B_3 \triangleq \{ \{  (i,j_i) ~:~ 0 \leq j_i \leq k-2, ~ 0 \leq i \leq k-1 \} ~:~ (j_0,j_1,\ldots,j_{k-1}) \in \cA  \}~,
$$
and
$$
B = B_1 \cup B_2 \cup B_3 ~.
$$
\hfill\quad $\blacksquare$
\end{construction}

Note, that set $B_1$ contains $k$ blocks, each one of size $k-1$. These blocks form one parallel class, i.e.,
a Steiner system S$(1,k-1,k(k-1))$.

The set $B_2$ contains $k-1$ blocks, each one of size $k$. These blocks form one parallel class, i.e.,
a Steiner system S$(1,k,k(k-1))$.

Since $\cA$ is an OA$(2,k,k)$ in which each one of the first $k$ rows has $k$ identical symbols, it follows that
in each one of the other rows each symbol between 0 and $k-1$ appears exactly once.
The blocks of the set $B_3$ are formed from these rows of $\cA$, where the symbol $k-1$ is omitted.
Let $\cA'$ denote the array constructed from these rows of $\cA$ omitting the symbol $k-1$.
It can be proved (see~\cite{Etz22}) that in $B_3$ there are $(k-1)k$ blocks, each block of size $k-1$, and they can be
partitioned into $k-1$ parallel classes, each one with $k$ blocks.

\begin{lemma}
\label{lem:pairs_base}
Each pair of $Q = \Z_k \times \Z_{k-1}$ is contained in exactly one block of $\cS$.
\end{lemma}
\begin{proof}
Note that the $i$-th column of $\cA$ is associated with the points $\{i\} \times \Z_{k-1}$
in the blocks of $B_3$ and the element $j \in \{0,1,\ldots,k-2\}$
in column $i$ of $\cA$ ($k-1$ is omitted) is associated with the point $(i,j)$. Therefore, since each ordered pair $(\alpha,\beta)$,
$\alpha , \beta \in \Z_{k-1}$ appears exactly once in each ordered pair of columns $(\ell,m)$ of $\cA$, it follows that each pairs of points
$\{(\ell,\alpha),(m,\beta)\}$ is contained in exactly one block of $\cS$, either in a block of $B_2$ when $\alpha=\beta$
or a block of $B_3$ when $\alpha \neq \beta$.
Each one of the remaining pairs of the form $\{(m,\alpha),(m,\beta)\}$, where $0 \leq m \leq k-1$ and
$0 \leq \alpha < \beta \leq k-2$, is contained exactly once in the blocks of $B_1$.
\end{proof}

Now, by Construction~\ref{const:fromLS} (where the blocks of $B_2$ take the place of $\cR$ and $B_1$ and the $k-1$ parallel classes of $B_3$
take the place of the $\cT_i$'s) and Theorem~\ref{thm:fromLS} we have the following consequence.
\begin{lemma}
\label{lem:applyCON}
By applying Construction~\ref{const:fromLS} on the blocks of $\cS$ we obtain an \textup{MS}$(2,k,\Z_2^{(k-1)k} \times \Z_{k+1})$.
\end{lemma}

Note that an \textup{MS}$(2,k,\Z_2^{(k-1)k} \times \Z_{k+1})$ was also constructed in Construction~\ref{const:fromOA} (see Theorem~\ref{thm:MSfromOA}).

\begin{example}
For $k=4$, the two set $B_1$ and $B_2$ contain the following blocks (codewords):
$$
B_1 =
\begin{array}{cccc}
111 & 000 & 000 & 000 \\
000 & 111 & 000 & 000 \\
000 & 000 & 111 & 000 \\
000 & 000 & 000 & 111 \\
\end{array}, ~~
B_2 =
\begin{array}{cccc}
100 & 100 & 100 & 100 \\
010 & 010 & 010 & 010 \\
001 & 001 & 001 & 001 \\
\end{array}~.
$$
The blocks of $B_1$ form a Steiner system \textup{S}$(1,3,12)$ and the blocks of $B_2$ form a Steiner system \textup{S}$(1,4,12)$.
The $12 \times 4$ orthogonal array \textup{OA}$(2,4,4)$ $\cA$, its derived array $\cA'$, its associated three parallel classes of $B_3$,
each one is a Steiner system \textup{S}$(1,3,12)$
and the system $\cS'$ derived from it via Construction~\ref{const:fromLS} are as follows:
$$
\cA=
\begin{array}{c}
0  0  0  0 \\
1  1  1  1 \\
2  2  2  2 \\
3  3  3  3 \\
3  0  2  1 \\
3  1  0  2 \\
3  2  1  0 \\
0  3  1  2 \\
2  3  0  1 \\
1  3  2  0 \\
0  2  3  1 \\
2  1  3  0 \\
1  0  3  2 \\
2  0  1  3 \\
0  1  2  3 \\
1  2  0  3 \\
\end{array},~
\cA'=
\begin{array}{cccc}
\varnothing & 0 & 2 & 1 \\
0 & \varnothing & 1 & 2 \\
2 & 1 & \varnothing & 0 \\
1 & 2 & 0 & \varnothing \\
\hline
\varnothing & 1 & 0 & 2 \\
1 & \varnothing & 2 & 0 \\
0 & 2 & \varnothing & 1 \\
2 & 0 & 1 & \varnothing \\
\hline
\varnothing & 2 & 1 & 0 \\
2 & \varnothing & 0 & 1 \\
1 & 0 & \varnothing & 2 \\
0 & 1 & 2 & \varnothing \\
\end{array},~
B_3 =
\begin{array}{cccc}
000 & 100 & 001 & 010 \\
100 & 000 & 010 & 001 \\
001 & 010 & 000 & 100 \\
010 & 001 & 100 & 000 \\
\hline
000 & 010 & 100 & 001 \\
010 & 000 & 001 & 100 \\
100 & 001 & 000 & 010 \\
001 & 100 & 010 & 000 \\
\hline
000 & 001 & 010 & 100 \\
001 & 000 & 100 & 010 \\
010 & 100 & 000 & 001 \\
100 & 010 & 001 & 000 \\
\end{array},~
\cS'=
\begin{array}{ccccc}
100 & 100 & 100 & 100 & 0 \\
010 & 010 & 010 & 010 & 0 \\
001 & 001 & 001 & 001 & 0 \\
111 & 000 & 000 & 000 & 1 \\
000 & 111 & 000 & 000 & 1 \\
000 & 000 & 111 & 000 & 1 \\
000 & 000 & 000 & 111 & 1 \\
000 & 100 & 001 & 010 & 2 \\
100 & 000 & 010 & 001 & 2 \\
001 & 010 & 000 & 100 & 2 \\
010 & 001 & 100 & 000 & 2 \\
000 & 010 & 100 & 001 & 3 \\
010 & 000 & 001 & 100 & 3 \\
100 & 001 & 000 & 010 & 3 \\
001 & 100 & 010 & 000 & 3 \\
000 & 001 & 010 & 100 & 4 \\
001 & 000 & 100 & 010 & 4 \\
010 & 100 & 000 & 001 & 4 \\
100 & 010 & 001 & 000 & 4 \\
\end{array}.
$$
\hfill\quad $\blacksquare$
\end{example}

We continue to develop a new construction. The system $\cS$ which was constructed will be called {\bf \emph{the base system}}.
Let $\cT$ be a resolvable S$(2,k,n)$ on the set of points $\Z_n$. Such a system has $\frac{n-1}{k-1}$ parallel classes.
We construct a new system $\cT(\cS)$ on the set of points $\Z_n \times \Z_{k-1}$.
Given any block $X=\{ x_1 , x_2, \ldots , x_k \}$ in $\cT$, we construct the base system on the set of points $X \times \Z_{k-1}$ which
are contained in $\Z_n \times \Z_{k-1}$. We analyze now the constructed system $\cT(\cS)$.

Each $(k-1)$-subset constructed from the set $B_1$ on the points $\{ i \} \times \Z_{k-1}$ is shared
by all the base systems on each block $X=\{ x_1 , x_2, \ldots , x_k \}$ of $\cT$ which contains the point $i$.
Therefore, the blocks derived from $B_1$ contribute to $\cT(\cS))$ one parallel class with blocks of size $k-1$ to the system $\cT(\cS)$.

For each parallel class of $\cT$, the blocks derived from $B_2$
(of all the blocks of the parallel class) yield one parallel class for $\cT(\cS)$ with a total of $\frac{n}{k}$ blocks,
each one of size $k$, for each block of the parallel class.
Since there are $\frac{n-1}{k-1}$ parallel classes in $\cT$, it follows that the blocks of $B_2$ contribute
$\frac{n-1}{k-1}$ parallel classes with blocks of size $k$ to $\cT(\cS)$.
Each such parallel class is a Steiner system S$(1,k,n(k-1))$.

The set $B_3$ contributes to the base system $k-1$ parallel classes. The system $\cT$ has $\frac{n-1}{k-1}$ parallel classes
and each parallel class yields $k-1$ parallel classes from the blocks of the base system. Hence, there are $n-1$ parallel classes
in $\cT(\cS)$ which are derived from the blocks of $B_3$ in the base system. Each parallel class has $n$ blocks, each one of size $k-1$,
i.e., each parallel class is a Steiner system S$(1,k-1,n(k-1))$.

The following lemma can be easily verified in a similar way to the proof of Lemma~\ref{lem:pairs_base}.
\begin{lemma}
Each pair of $\Z_n \times \Z_{k-1}$ is contained exactly once in $\cT(\cS)$.
\end{lemma}

Now, by applying Construction~\ref{const:fromLS} and Theorem~\ref{thm:fromLS} as befor
(see Lemma~\ref{lem:applyCON}) we have the following consequence.
\begin{lemma}
By applying Construction~\ref{const:fromLS} on the blocks of $\cT(\cS)$ we obtain a mixed
Steiner system \textup{MS}$(2,k,\Z_2^{(k-1)n} \times \Z_{n+1})$.
\end{lemma}

We continue to produce mixed Steiner systems for more parameters by modifying the system $\cT(\cS)$.
The idea is to replace some of the $n$ parallel classes with blocks of size $k-1$, derived from $B_3$, with parallel classes
with blocks of size $k$. The construction starts by choosing a second system instead of the base system. This system will be called
{\bf \emph{the replace system}}. The replace system has the blocks of the set $B_1$, but instead of the two sets $B_2$
and $B_3$ we are using another set of blocks $B_4$. The set $B_4$ is derived from an orthogonal array OA$(2,k,k-1)$.
By Theorem~\ref{thm:LS_OA} it
immediately implies that we can apply the construction whenever $k-1$ and $k$ are both powers of primes.
Let $\cD$ be this orthogonal array and
$$
B_4 \triangleq \{ \{  (i,j_i) ~:~ 0 \leq j_i \leq k-2, ~ 0 \leq i \leq k-1 \} ~:~ (j_0,j_1,\ldots,j_{k-1}) \in \cD  \}.
$$
We note that $B_4$ is different from $B_3$ since $B_3$ is constructed from an OA$(2,k,k)$ while $B_4$ is constructed
from an OA$(2,k,k-1)$. Moreover, $B_3$ is constructed from the orthogonal array $\cA$ by omitting some rows and also omitting the symbol $k-1$,
while $B_4$ is constructed from all the rows and all the $k-1$ symbols of the orthogonal array $\cD$.

Now in the new construction of a system $\cT'(\cS)$ similar to $\cT(\cS)$ for each parallel class $\cP$ of $\cT$
we have to make a choice if for each block
$X=\{ x_1 , x_2, \ldots , x_k \}$ of $\cP$, either to use the base system or to use the replace system for
the points of $X \times \Z_{k-1}$. There are $\frac{n-1}{k-1}$ parallel classes in $\cT$, each one contributes $k-1$
parallel systems associated with $B_3$ to $\cT(\cS)$. This implies the following theorem.

\begin{theorem}
\label{thm:manyparam}
The blocks of $\cT'(\cS)$, yield an \textup{MS}$(2,k,\Z_2^{(k-1)n} \times \Z_{n+1-(k-1) \cdot i})$, for
each $0 \leq i \leq \frac{n-1}{k-1}$.
\end{theorem}

Note, that when $i= \frac{n-1}{k-1}$ the mixed Steiner system is a Steiner system S$(2,k,(k-1)n+1)$.

\begin{example}
For $k=4$ the following is a $9 \times 4$ orthogonal array OA$(2,4,3)$ with its blocks for the replace system.
$$
\cD=
\begin{array}{cccc}
0 & 0 & 0 & 0 \\
0 & 1 & 1 & 1 \\
0 & 2 & 2 & 2 \\
1 & 0 & 1 & 2 \\
1 & 1 & 2 & 0 \\
1 & 2 & 0 & 1 \\
2 & 0 & 2 & 1 \\
2 & 1 & 0 & 2 \\
2 & 2 & 1 & 0 \\
\end{array}~~,
\begin{array}{cccc}
100 & 100 & 100 & 100 \\
100 & 010 & 010 & 010 \\
100 & 001 & 001 & 001 \\
010 & 100 & 010 & 001 \\
010 & 010 & 001 & 100 \\
010 & 001 & 100 & 010 \\
001 & 100 & 001 & 010 \\
001 & 010 & 100 & 001 \\
001 & 001 & 010 & 100 \\
\end{array}~.
$$
\hfill\quad $\blacksquare$
\end{example}

Finally, we note that Theorem~\ref{thm:manyparam} depends on the existence of a two power primes $k$ and $k-1$.
Such power primes exist only when
$k-1$ is a Mersenne prime, or $k$ is a Fermat prime, or for $k=9$.


\section*{Acknowledgement}

Tuvi Etzion thanks
Lijun Ji~\cite{Ji25} that have noted that some results in a related draft were already known under different framework.
The references he provided led to some important parts of the current draft.

\end{document}